\documentclass[11pt]{amsart}
\usepackage{graphicx,amssymb,amsmath,amsthm}
\usepackage{enumerate}
\usepackage{comment,cite,color}
\usepackage{cite,color}
\usepackage{algorithmic}
\usepackage{mathrsfs}
\usepackage[ruled]{algorithm}
\usepackage{epsfig}
\usepackage{enumitem}

\newcommand{\innerp}[1]{\langle {#1} \rangle}

\newcommand{\sign}{{\rm sign}}

\newcommand{\R}{{\mathbb R}}

\newcommand{\C}{{\mathbb C}}

\renewcommand{\eqref}[1]{(\ref{#1})}

\newcommand{\mhsp}{\hspace{2em}}

\renewcommand{\H}{{\mathbb H}}

\newtheorem{definition}{Definition}[section]
\newtheorem{corollary}{Corollary}[section]
\newtheorem{theorem}{Theorem}[section]
\newtheorem{lemma}{Lemma}[section]
\newtheorem{remark}{Remark}[section]
\newcommand{\zz}{^{\top}}
\date{}
\begin{document}
\bibliographystyle{plain}
\title{The $ \ell_1 $-analysis with redundant dictionary in phase retrieval}

\author{Bing Gao}
\address{LSEC, Inst.~Comp.~Math., Academy of
Mathematics and System Science,  Chinese Academy of Sciences, Beijing, 100091, China}
\email{gaobing@lsec.cc.ac.cn}

\begin{abstract}
This article presents new results concerning the recovery of a signal from magnitude only measurements where the signal is not sparse in an orthonormal basis but in a redundant dictionary. To solve this phaseless problem, we analyze the $ \ell_1 $-analysis model. Firstly we investigate the noiseless case with presenting a null space property of the measurement matrix under which the $ \ell_1 $-analysis model provide an exact recovery. Secondly  we introduce a new property (S-DRIP) of the measurement matrix. By solving the $ \ell_1 $-analysis model, we prove that this property can guarantee a stable recovery of real signals that are nearly sparse in highly overcomplete dictionaries.
\end{abstract}
\maketitle{}
\vspace{0.3cm}

\section{Introduction}
\subsection{Phase Retrieval}
Phase retrieval is the process of recovering signals from phaseless measurements. It is of fundamental importance in numerous ares of applied physics and engineering \cite{JF},\cite{DL}. In general form, phase retrieval problem is to estimate the original signal $ x_0\in\mathbb{H}^n $ ($\mathbb{H}=\C $ or $ \R $) from
\begin{equation}\label{primemodel}
|Ax|=|Ax_0|+e,
\end{equation}
where $A = [a_1,\ldots, a_m]\zz\in\H^{m\times n} $ is the measurement matrix and $ e=[e_1,\cdots, e_m]\in\H^m $ is an error term. While only the magnitude of $ Ax_0 $ is available, it is important to note that the setup naturally leads to ambiguous solutions. For example, if $ \hat{x}\in\H^n $ is a solution to (\ref{primemodel}), then any multiplication of $ \hat{x} $ and a scalar $ c\in\H $ ($ |c|=1 $) is also a solution to (\ref{primemodel}). Hence, these global ambiguities are considered acceptable for this problem. In this paper, we recover the signal $ x_0 $ actually means that we reconstruct $ x_0 $ up to a unimodular constant.
 
It is known that, when $ \H=\R $, at least $ 2n-1 $ measurements are needed to recover a signal $ x\in\R^n$ \cite{RPD}. For the complex case, the minimum number of measurements are proved to be at least $ 4n-4 $ when $ n $ is in the form of $ n=2^k+1, k\in\mathbb{Z_+} $ \cite{CDMC}. However, for a general dimension $ n $, the same question is still open. About the minimum number of observations, more details can be found in \cite{BN}, \cite{YZ}. To reduce the measurement numbers, priori information must be given. The most common priori information is sparsity, which means that only few elements in the target signal $ x_0 $ is nonzero. Here we say a signal is $ k $-sparse if there are at most $ k $ non-zero elements in the signal. In view of sparse signals, phase retrieval is also known as compressive phase retrieval, which have many applications in data acquisition \cite{cpr_J},\cite{cpr_E}. The compressive phase retrieval problem is in fact the magnitude-only compressive sensing problem. For this problem, Wang and Xu explored the minimum number of measurements and extended the null space property in compressed sensing to phase retrieval \cite{YZ}. In \cite{VX14}, Voroniski and Xu gave the definition of \textit{strong restricted isometry property} (Definition \ref{srip}) and then many conclusions in compressed sensing can be extended to compressive phase retrieval, such as instance optimality \cite{BYZ}. The above conclusions hold just for signals which are sparse in the standard coordinate basis. However, there are many examples in which a signal of interest is not sparse in an orthonormal basis but sparse in some transform basis. In resent years, many researchers laid special stress on analysing these dictionary-sparse signals in compressed sensing \cite{CEN10}, \cite{AXA}, \cite{HKP}. However, the phase retrieval literature is lacking on this subject. We will focus on this problem in this paper.
\subsection{The $ \ell_1 $-analysis with redundant dictionary}
At normal state, sparsity is expressed not in terms of an orthonormal basis but in terms of an overcomplete dictionary. That is to say, the signal $ x_0\in\H^n $ can be expressed as $ x_0=Dz $, where $ D\in\H^{n\times N } $ is a frame and $ z\in\H^N $ is a sparse vector. 
In this paper, we use $ D^* $ to represent the adjoint conjugate of $ D $ when $ \H=\C $, while when $ \H = \R $, we use $ D^* $ to represent the transpose of $ D $.

In compressed sensing, to reconstruct the signal $ x_0 $, the most commonly used model is the $ \ell_1 $-analysis model
\begin{equation}\label{l1_com-sen}
\min\|D^*x\|_1\mhsp \mbox{subject to} \mhsp \|Ax-Ax_0\|_2^2\leq\epsilon^2,
\end{equation}
where $ \epsilon $ is the upper bound of the noise.
In \cite{CEN10}, Cand\`{e}s, Eldar, Needell and Randall proved that when $ D $ is a tight frame and $ D^*x_0 $ is almost $ k $-sparse,  the $ \ell_1 $-analysis (\ref{l1_com-sen}) can guarantee a stable recovery provided that the measurement matrix is Gaussian random matrix with $ m=\mathcal{O}(k\log(n/k)) $.  

For the phase retrieval problem, we also analyze the $ \ell_1 $-analysis model
\begin{equation}\label{l1analysis}
\min\|D^*x\|_1\mhsp \mbox{subject to} \mhsp \||Ax|-|Ax_0|\|_2^2\leq\epsilon^2,
\end{equation}
where $ \epsilon $ is the upper bound of the noise level.

In this paper, we aim to explore the conditions under which the $ \ell_1 $-analysis model (\ref{l1analysis}) can generate an accurate or a stable solution to (\ref{primemodel}).
Note that when $ D=I $, this problem is reduced to the traditional phase retrieval and the $ \ell_1 $-analysis model is reduced to 
\begin{equation}\label{l1_ph-re}
\min\|x\|_1\mhsp \mbox{subject to} \mhsp \||Ax|-|Ax_0|\|_2^2\leq\epsilon^2.
\end{equation}
For this case, when $ \H=\R $, Gao, Wang and Xu provided a detailed analysis of (\ref{l1_ph-re}) in \cite{BYZ} and had the conclusion that a $ k $-sparse signal can be stably recovered by $ \mathcal{O}(k\log(n/k)) $ Gaussian random measurements. Then a natural question that comes to mind is whether this conclusion still holds for a general frame $ D $.
\subsection{Organization}
The rest of the paper is organized as follows. In section 2, we give notations and recall some previous conclusions. In section 3, for noiseless case ($ \epsilon=0 $), we analyze the null space of the measurement matrix and give sufficient and necessary conditions for (\ref{l1analysis}) to achieve an exact solution, which will be discussed in real and complex case separately.  In general, it's hard to check whether a matrix satisfies the null space property or not. So in section 4, we introduce a new property (S-DRIP) (Definition \ref{drip}) on the measurement matrix, which is a natural generalization of the DRIP (see \cite{CEN10} for more details). Using this property, we prove that when the measurement matrix is real Gaussian random matrix with $ m\geq\mathcal{O}(k\log(n/k)) $ the $ \ell_1 $-analysis (\ref{l1analysis}) can guarantee a stable recovery of real signals which are $ k $-sparse under a redundant dictionary. Last, some proofs are given in the Appendix.

\section{Notations and previous results}
Throughout this paper, we use $ D\in\H^{n\times N} $ as a frame with full column rank.  
Let 
\[
\Sigma_k^N :=\{ x\in\H^N: \|x\|_0\leq k \}
\]
and
$$
D\Sigma_k^N :=\{x\in \H^{n}: \exists z\in\Sigma_k^N, x=Dz\}.
$$
Suppose the target signal $x_0$ is in the set $ D\Sigma_k^N $,
which means that $x_0$ can be represented as $x_0=Dz_0$, where $z_0\in \Sigma_k^N$. 

The best $ k $-term approximation error is defined as
\[
\sigma_k(x)_1 := \min_{z\in\Sigma_k}\|x-z\|_1.
\]
We use $ [m] $ to represent the set $ \{1,2,\ldots,m \} $. Suppose $ T\subseteq[m] $ is a subset of $ [m] $. We use $ T^c $ to represent the complement set of $ T $ and $ |T| $ to denote the cardinal number of $ T $. Let $ A_T:=[a_j, j\in T]\zz $ denote the sub-matrix of $ A $ where only rows with indices in $ T $ are kept. Denote $ \mathcal{N}(A) $ as the null space of $ A $.
\begin{definition}[DRIP]\cite{CEN10}
Fix a dictionary $ D\in\R^{n\times N} $ and a matrix $ A\in\R^{m\times n} $. The matrix $ A $ satisfies the DRIP with parameters $ \delta $ and $ k $ if
\[
(1-\delta)\|Dz\|_2^2\leq\|ADz\|_2^2\leq(1+\delta)\|Dz\|_2^2
\]
holds for all $ k $-sparse vectors $ z\in\R^N $.
\end{definition}
The paper \cite{CEN10} shows that the Gaussian matrices and other random compressed sensing matrices satisfy the DRIP of order $ k $ provided the number of measurements $ m $ on the order of $ \mathcal{O}(k\log(n/k)) $.
\begin{definition}[SRIP]\cite{VX14}\label{srip}
We say the matrix $A=[a_1,\cdots,a_m]\zz \in\mathbb{R}^{m\times n}$  has the Strong Restricted Isometry Property of order $ k $ and constants $\theta_-,\ \theta_+\in (0, 2)$ if
$$
\theta_-\|x\|_2^2\leq \min_{ |I|\subseteq[m], |I|\geq m/2}\|A_Ix\|_2^2\leq\max_{|I|\subseteq[m],|I|\geq m/2}
\|A_Ix\|_2^2\leq\theta_+\|x\|_2^2
$$
holds for all $  k $-sparse signals $x\in\mathbb{R}^n$.
\end{definition}
This property was first introduced in \cite{VX14}. Voroninski and Xu also proved that the Gaussian random matrices satisfy SRIP with high probability. More details can be found in the following theorem.
\begin{theorem}{\cite{VX14}}\label{gaussiansrip}
Suppose that $t>1$ and $ A\in \mathbb{R}^{m\times n} $ is a random Gaussian matrix with $m=\mathcal{O}(tk\log(n/k))$. Then there exist $\theta_-$,  $\theta_+$, with $0<\theta_-<\theta_+<2$, such that $ A $ satisfies SRIP of order $ tk $ and constants  $\theta_-$,   $\theta_+$, with probability $1-exp(-cm/2)$, where $c>0$ is an absolute constant and $\theta_-$,  $\theta_+$ are independent with $ t $.
\end{theorem}

\section{The Null Space Property}
In this section, for $ x_0\in D\Sigma_k^N $, we consider the noiseless situation
\begin{equation}\label{ nsp model}
\min\|D^*x\|_1\mhsp \mbox{subject to} \mhsp |Ax|=|Ax_0|.
\end{equation}
Similarly as the traditional compressed sensing problem, we analyze the null space of the measurement matrix $ A $ to explore conditions under which (\ref{ nsp model}) can obtain $ cx_0 $ ($ |c|=1 $).
\subsection{The Real Case} 
We first restrict the signals and measurements to the field of real numbers. The next theorem provides a sufficient and necessary condition for the exact recovery of (\ref{ nsp model}).
\begin{theorem}\label{NSPreal}
For given matrix $ A\in\R^{m\times n} $ and dictionary $ D\in\R^{n\times N} $, we claim that the following properties are equivalent.
\begin{flushleft}
(A) For any $x_0\in  D\Sigma_k^N$,
$$
\textup{argmin}_{x\in\R^n}\{\|D^*x\|_1:|Ax|=|Ax_0|\}=\{\pm x_0\}.
$$
(B) For any $T\subseteq [m]$, it holds 
$$
\| D^*(u+v)\|_1<\| D^*(u-v)\|_1
$$
for all
$$
u\in \mathcal{N}(A_T)\backslash\{0\} ,\quad  v\in \mathcal{N}(A_{T^c})\backslash\{0\}
$$
satisfying
$$
u+v\in D\Sigma_k^N.
$$
\end{flushleft}
\end{theorem}
\begin{proof}
(B)$\Rightarrow$(A).
Assume (A) is false, namely, there exists a solution $ \hat x\neq \pm x_0 $ to (\ref{ nsp model}).
As $ \hat{x} $ is a solution, we have
\begin{align}\label{realcon1}
|A\hat x|=|Ax_0|
\end{align}
and
\begin{align}\label{realcon2}
\| D^*{\hat x}\|_1\leq\| D^*x_0\|_1.
\end{align}
Denote $a_j\zz, j=1,\ldots,m$ as the rows of $ A $.
Then (\ref{realcon1}) implies that there exists a subset $T\subseteq[m]$ satisfying
\[
j\in T, \quad \langle a_j, x_0+\hat{x}\rangle=0,
\]
\[
j\in T^c, \quad \langle a_j, x_0-\hat{x}\rangle=0,
\]
i.e.,
$$
 A_T(x_0+\hat x)=0, \quad A_{T^c}(x_0-\hat x)=0.
$$
Define
$$
 u:=x_0+\hat x,\quad v:=x_0-\hat x.
$$
As $ \hat{x}\neq\pm x_0 $, we have $u\in \mathcal{N}(A_T)\backslash\{0\}$, $v\in \mathcal{N}(A_{T^c})\backslash \{0\}$ and $u+v=2x_0\in D\Sigma_k^N$.
Then from (B), we know
$$
 \| D^*x_0\|_1<\| D^*\hat x\|_1,
$$
which contradicts with (\ref{realcon2}).

(A)$\Rightarrow$(B). Assume (B) is false, which means that there exists a subset $T\subseteq [m]$,
\begin{equation}\label{zeroterm}
u\in \mathcal{N}(A_T)\backslash\{0\},\quad v\in \mathcal{N}(A_{T^c})\backslash \{0\},
\end{equation}
such that
$$
u+v\in D\Sigma_k^N
$$
and
\begin{equation}\label{real_1}
\| D^*(u+v)\|_1\geq \| D^*(u-v)\|_1.
\end{equation}
Let $ x_0:=u+v\in D\Sigma_k^N$ be the signal we want to recover. Set $\tilde{x} :=u-v$ and we have $ \tilde{x}\neq\pm x_0 $. Then from (\ref{real_1}) we have
\begin{equation}\label{realproof1}
\| D^*\tilde {x}\|_1\leq\| D^*x_0\|_1.
\end{equation}
Let $a_j\zz, j=1,\ldots,m$ denote the rows of $ A $. Then from the definition of $ x_0 $ and $ \tilde{x} $, we have
$$
 2\langle a_j,u \rangle=\langle a_j,x_0+\tilde{x}\rangle,
$$
$$
2\langle a_j,v\rangle=\langle a_j,x_0-\tilde{x}\rangle.
$$
By (\ref{zeroterm}), the subset $ T $ satisfies
$$
j\in T,\quad  \langle a_j,x_0\rangle=-\langle a_j,\tilde{x} \rangle
$$
and
$$
j\in T^c, \quad \langle a_j,x_0\rangle=\langle a_j,\tilde{x} \rangle,
$$
which implies
\begin{equation}\label{realproof2}
|Ax_0|=|A\tilde{x}|.
\end{equation}
Putting (\ref{realproof1}) and (\ref{realproof2}) together, we know $ \tilde{x} $ is a solution to model (\ref{ nsp model}). However, $\tilde{x}\neq\pm x_0 $ contradicts with (A).
\end{proof}
\subsection{The Complex Case}
We now consider the same problem in complex case which means that the signals and measurements are all in the complex number field. We say that $ \mathcal{S}=\{S_1,\ldots, S_p\} $ is any partition of $ [m] $ if 
\[
S_j\in[m],\,\, \bigcup_{j=1}^{p} S_j=[m] \,\,\text{and}\,\, S_l\cap S_j=\emptyset, \forall\, l\neq j.
\]
Set $ \mathbb{S}:=\{ c\in \C, |c|=1\} $.
The next theorem is a generalization of Theorem \ref{NSPreal}.
\begin{theorem}\label{NSPcomp}
For given matrix $ A\in\C^{m\times n} $ and dictionary $ D\in\C^{n\times N} $, we claim that the following properties are equivalent.
\begin{flushleft}
(A) For any given $x_0\in D\Sigma_k^N$,
$$
\textup{argmin}_{x\in\C^n}\{\|D^*x\|_1:|Ax|=|Ax_0|\}=\{cx_0, c\in\mathbb{S}\}.
$$
(B) Suppose $ S_1,\ldots,S_p$ is any partition of $[m]$. For any given
$\eta_j\in \mathcal{N}(A_{S_j})\backslash\{0\}$, if
\begin{equation}\label{comp_nsp_condition}
\frac{\eta_1-\eta_l}{c_1-c_l}=\frac{\eta_1-\eta_j}{c_1-c_j} \in D\Sigma_k^N\backslash\{0\},\,\,j,l \in[2:p],\,\,j\neq l
\end{equation}
holds for some pairwise distinct $c_1,\ldots,c_p\in\mathbb{S}$, we have
$$
\| D^*(\eta_j-\eta_l)\|_1<\| D^*(c_l\eta_j-c_j\eta_l)\|_1.
$$
\end{flushleft}
\end{theorem}
\begin{proof}
$(B)\Rightarrow(A)$. Suppose the statement (A) is false. That is to say, there exists a solution $ \hat{x}\notin \{cx_0, c\in\mathbb{S}\}$ to (\ref{ nsp model}) which satisfies
\begin{equation}\label{compconf1}
\| D^*\hat x\|_1\leq\| D^*x_0\|_1
\end{equation}
and
\begin{equation}\label{complexproof1}
|Ax_0|=|A\hat x|.
\end{equation}
Denote $a_j^*, j=1,\ldots,m$ as the rows of $ A $. From (\ref{complexproof1}) we have
$$
\langle a_j,c_jx_0\rangle=\langle a_j,\hat x\rangle,
$$
with $ c_j\in\mathbb{S}, \,j=1,\ldots, m  $. We can define an equivalence relation on $[m]$, namely $j\sim l$, when $c_j=c_l$. This equivalence relation leads to a partition $\mathcal{S}=\{S_1,\ldots,S_p\}$ of $[m]$. For any $S_j$, we have
$$
A_{S_j}(c_jx_0)=A_{S_j}\hat x.
$$
Set $\eta_j:=c_jx_0-\hat x$. Then we have $\eta_j\in \mathcal{N}(A_{S_j})\backslash\{0\}$ and
$$
\frac{\eta_1-\eta_l}{c_1-c_l}=\frac{\eta_1-\eta_j}{c_1-c_j}=x_0\in D\Sigma_k^N\,\,\,\,\forall j,l \in[2:p],\,\,j\neq l.
$$
By the condition (B), we can get
$$
\| D^*(\eta_j-\eta_l)\|_1<\| D^*(c_l\eta_j-c_j\eta_l)\|_1,
$$
i.e.,
$$
\| D^*(c_j-c_l)x_0\|_1<\| D^*(c_j-c_l)\hat x\|_1.
$$
That is equivalence to
$$
\| D^*x_0\|_1<\| D^*\hat x\|_1,
$$
which contradicts with (\ref{compconf1}).

$(A)\Rightarrow(B)$. Assume (B) is false, namely, there exists a partition $\mathcal{S}=\{S_1,\ldots,S_p\}$ of $[m]$, $\eta_j\in \mathcal{N}(A_{S_j})\backslash\{0\}$, $ j\in[1:p] $ and some pairwise distinct $ c_1,\ldots,c_p\in\mathbb{S} $ satisfying (\ref{comp_nsp_condition}) but
$$
\| D^*(\eta_{j_0}-\eta_{l_0})\|_1\geq\| D^*(c_{l_0}\eta_{j_0}-c_{j_0}\eta_{l_0})\|_1
$$
holds for some distinct $j_0, l_0\in [1:p]$.
Set
$$
\tilde{x}:=c_{l_0}\eta_{j_0}-c_{j_0}\eta_{l_0}, \,\, c_{l_0}\neq c_{j_0},
$$
$$
x_0:=\eta_{j_0}-\eta_{l_0}\in D\Sigma_k^N.
$$
Then we have
\[
\tilde{x}\notin\{cx_0, c\in\mathbb{S}\}
\]
and 
\begin{equation}\label{complexproof2}
\| D^*\tilde{x}\|_1\leq\| D^*x_0\|_1.
\end{equation}
Let $a_j^*, j=1,\ldots,m$ denote the rows of $ A $. From $ \eta_j\in \mathcal{N}(A_{S_j})\backslash\{0\}  $, we have
\begin{align*}
\langle a_k,\eta_{j_0}\rangle=0\,\,\, \text{and}\,\, \, \langle a_k,\eta_{l_0}\rangle=0, \quad \forall  k\in S_{l_0}\cup S_{j_0}.
\end{align*}
The definition of $ x_0 $ and $ \tilde{x} $ implies
\begin{equation}\label{complexproof3}
|\langle a_k,x_0\rangle|=|\langle a_k,\tilde{x}\rangle|, \,\,\forall\,\,  k\in S_{l_0}\cup S_{j_0}.
\end{equation}
For $k\notin S_{l_0}\cup S_{j_0} $, we might as well suppose $k\in S_t $ $ (t\neq l_0, j_0)$, i.e., $ \langle a_k, \eta_t\rangle = 0 $.  From
$$
\frac{\eta_1-\eta_l}{c_1-c_l}=\frac{\eta_1-\eta_j}{c_1-c_j},
$$
we can get
$$
\frac{\eta_j-\eta_l}{c_j-c_l}=\frac{\eta_m-\eta_n}{c_m-c_n}, \,\,\text{for any}\,\,j,l,m,n \,\,\text{are distinct integers}.
$$
Set
$$
y_0:= \frac{\eta_{j_0}-\eta_t}{c_{j_0}-c_t}=\frac{\eta_{l_0}-\eta_t}{c_{l_0}-c_t}.
$$
Then we have
$$
\eta_{j_0}=(c_{j_0}-c_t)y_0+\eta_t,
$$
$$
\eta_{l_0}=(c_{l_0}-c_t)y_0+\eta_t.
$$
So $ \tilde{x} $ and $ x_0 $ can be rewritten as 
$$
\tilde{x}=c_{l_0}\eta_{j_0}-c_{j_0}\eta_{l_0}=c_t(c_{j_0}-c_{l_0})y_0+(c_l-c_j)\eta_t,
$$
$$
x_0=\eta_{j_0}-\eta_{l_0}=(c_{j_0}-c_{l_0})y_0.
$$
Then $\langle a_k,\eta_t\rangle=0$ implies
\begin{equation*}
|\langle a_k,\tilde{x}\rangle|=|\langle a_k,x_0\rangle|,\quad \text{for}\,\, k\in S_t.
\end{equation*}
Using a similar argument, we can prove that the claim is also true for other subset $ S_j $. So we have
\begin{equation}\label{complexprooof4}
|\langle a_k,\tilde{x}\rangle|=|\langle a_k,x_0\rangle|, \quad \forall k.
\end{equation}
Combining (\ref{complexproof2}) and (\ref{complexprooof4}), we know $\tilde{x}$ is also a solution to (\ref{ nsp model}). However, $ \tilde{x}\notin\{cx_0, c\in\mathbb{S}\} $ contradicts with (A).
\end{proof}
\begin{remark}
If we chose $ D=I $, the null space property in Theorem \ref{NSPreal} and Theorem \ref{NSPcomp} is consistent with the null space property which was introduced in paper \cite{YZ}.
\end{remark}
By the Theorem \ref{NSPreal} and Theorem \ref{NSPcomp}, we know that it is possible to find a good measurement matrix to obtain an exact solution by solving the model (\ref{ nsp model}). But in general, condition (B) is difficult to be checked. So in section \ref{sdrip_sec}, we provide another property of the measurement matrix which can be satisfied by Gaussian random matrix. 

\section{S-DRIP and Stable Recovery}\label{sdrip_sec}
In compressed sensing, for any tight frame $ D $, \cite{CEN10} has the conclusion that a signal $ x_0\in D\Sigma_k^N $ can be approximately reconstructed using $ \ell_1 $-analysis (\ref{l1_com-sen}) provided the measurement matrix satisfies DRIP and $ D^*x_0 $ decays rapidly. While in phase retrieval, when $ \H=\R $, Gao, Wang and Xu proved that if the measurement matrix satisfies SRIP, then the $ \ell_1 $-analysis (\ref{l1_ph-re}) can provide a stable solution to traditional phase retrieval problem \cite{BYZ}. Next we combine this two results to explore the conditions under which the $ \ell_1 $-analysis model (\ref{l1analysis}) can guarantee a stable recovery.

We first impose a natural property on the measurement matrix, which is a combination of DRIP and SRIP.
\begin{definition}[S-DRIP]\label{drip}
Let $ D\in\R^{n\times N} $ be a frame. 
We say the measurement matrix $ A $ obeys the S-DRIP of order $ k $ with constants $\theta_-, \theta_+\in(0:2)$ if
$$
\theta_-\|Dv\|_2^2\leq \min_{ I\subseteq[m], |I|\geq m/2}\|A_IDv\|_2^2\leq\max_{I\subseteq[m],|I|\geq m/2}
\|A_IDv\|_2^2\leq\theta_+\|Dv\|_2^2
$$
holds for all $ k $-sparse signals $v\in\R^N$.
\end{definition}
Thus a matrix $ A\in\mathbb{R}^{m\times n} $ satisfying S-DRIP means that any $ m'\times n $ submatrix of $ A $, with $ m'\geq m/2 $ satisfies DRIP with appropriate parameters.

In fact any matrix $ A\in\R^{m\times n} $ obeying
\begin{equation}\label{forsrip}
\mathbb{P}[c_-\|Dv\|_2^2\leq\min_{ I\subseteq[m], |I|\geq m/2}\|A_IDv\|_2^2\leq\max_{I\subseteq[m],|I|\geq m/2}
\|A_IDv\|_2^2\leq c_+\|Dv\|_2^2]\geq 1-2e^{-\gamma m}
\end{equation}
($ 0<c_-<c_+<2 $ and $ \gamma $ is a positive number constant) for fixed $ Dv\in\R^n $ will satisfy the S-DRIP with high probability. This can be seen by a standard covering argument (see the proof of Theorem 2.1 in \cite{VX14}). In \cite{VX14}, Voroninski and Xu proved that Gaussian random matrix satisfies (\ref{forsrip}) in Lemma 4.4.
So we have the following conclusion.
\begin{corollary}\label{gauss_sdrip}
Gaussian random matrix $ A\in\R^{m\times n} $ with $ m=\mathcal{O}(tk\log(n/k)) $ satisfies the S-DRIP of order $ tk $ with constants $\theta_-, \theta_+\in(0:2)$. 
\end{corollary}
For $ x_0\in D\Sigma_k^N  $, we return to consider the solving model
\begin{equation}\label{l1 dictionary model}
\min\|D^*x\|_1\mhsp \mbox{subject to} \mhsp \||Ax|-|Ax_0|\|_2^2\leq\epsilon^2,
\end{equation}
where $ \epsilon $ is the error bound.
Here all signals and matrices are all restricted to the real number field.
The next theorem tells under what conditions the solution to (\ref{l1 dictionary model}) is stable.
\begin{theorem}\label{maintheorem}
Assume that $ D\in\R^{n\times N} $ is a tight frame and $ x_0\in D\Sigma_k^N $. The matrix $A\in \mathbb{R}^{m\times n}$ satisfies the S-DRIP of order $tk$ (t is a positive integer) and level $ \theta_-, \theta_+ \in (0:2)$, with
$$
t\geq\max\{\frac{1}{2\theta_--\theta_-^2},\frac{1}{2\theta_+-\theta_+^2}\}.
$$
Then the solution $ \hat{x} $ to (\ref{l1 dictionary model}) satisfies
$$
\min\{\|\hat{x}-x_0\|_2,\|\hat{x}+x_0\|_2\}\leq c_1\epsilon+c_2\frac{2\sigma_k(D^*x_0)_1}{\sqrt{k}},
$$
where
$ c_1=\frac{\sqrt{2(1+\delta)}}{1-\sqrt{t/(t-1)}\delta} $, $  c_2=\frac{\sqrt{2}\delta+\sqrt{t(\sqrt{(t-1)/t}-\delta)\delta}}{t(\sqrt{(t-1)/t}-\delta)}+1.$
Here $ \delta $ is a constant satisfying
\[
\delta\leq\max\{1-\theta_-, \theta_+-1  \}\leq\sqrt{\frac{t-1}{t}}.
\]
\end{theorem}
We first give a more general lemma, which is the key to prove Theorem \ref{maintheorem}.
\begin{lemma}\label{mainlemma}
Let $ D\in\mathbb{R}^{n\times N} $ be an arbitrary tight frame,
$ x_0\in D\Sigma_k^N$ and  $ \rho\geq 0$. Suppose that $ A\in\mathbb{R}^{m\times n} $ is a measurement matrix satisfying the DRIP with
$ \delta = \delta_{tk}^A\leq\sqrt{\frac{t-1}{t}} $ for some $ t>1 $.
Then for any
\[
D^*\hat{x}\in \{D^*x\in \mathbb{R}^N : \|D^*x\|_1\leq \|D^*x_0\|_1+\rho, \, \|Ax-Ax_0\|_2\leq\epsilon \},
\]
we have
\begin{equation*}
\|\hat{x}-x_0\|_2\leq c_1\epsilon+c_2\frac{2\sigma_k(D^*x_0)_1}{\sqrt{k}}+c_2\cdot\frac{\rho}{\sqrt{k}},
\end{equation*}
where
$ c_1=\frac{\sqrt{2(1+\delta)}}{1-\sqrt{t/(t-1)}\delta} $, $  c_2=\frac{\sqrt{2}\delta+\sqrt{t(\sqrt{(t-1)/t}-\delta)\delta}}{t(\sqrt{(t-1)/t}-\delta)}+1.$
\end{lemma}
We put the proof of this Lemma in the Appendix.
\begin{remark}
When $ D=I $, which corresponds to the case of standard compressed sensing, this result is consistent with Lemma 2.1 in \cite{BYZ}. 
\end{remark}
\begin{remark}
Here the DRIP constant is better than the constant given in \cite{Baker}. In \cite{Baker}, Baker established a generated DRIP constant for compressed sensing. He proved that signals with $ k $-sparse $ D $-representation can be reconstructed if the measurement matrix satisfies DRIP with constant $ \delta_{2k}<\frac{2}{3}$. We extended his approach to get a better bound $ \delta_{tk}\leq\sqrt{\frac{t-1}{t}} $. As \cite{CZ14} shows, in the special case $ D=I $, for any $ t\geq 4/3 $, the condition $ \delta_{tk}\leq\sqrt{\frac{t-1}{t}} $ is sharp for stable recovery in the noisy case. So it is not difficult to show that for any tight frame $ D $, the condition $ \delta_{tk}\leq\sqrt{\frac{t-1}{t}} $ is also sharp when $ t\geq 4/3 $.
\end{remark}
\begin{proof}[Proof of the Theorem \ref{maintheorem} ]
As $\hat{x}$ is the solution to (\ref{l1 dictionary model}), we have
\begin{equation}\label{option1}
    \|D^*\hat{x}\|_1\leq\|D^*x_0\|_1
\end{equation}
and
\begin{equation}\label{eq:option2}
   \||A\hat{x}|-|Ax_0|\|_2^2\leq\epsilon^2.
\end{equation}
 Denote $ a_j\zz, j\in\{1,\ldots,m\} $ as the rows of $ A $ and divide $\{1,\ldots,m\}$ into two groups:
\[
T=\{j\mid \sign(\innerp{a_j,\hat{x}})=\sign(\innerp{a_j,x_0})\},
\]
\[
T^c=\{j\mid\sign(\innerp{a_j,\hat{x}})=-\sign(\innerp{a_j,x_0})\}.
\]
Then either $|T|\geq m/2$ or  $|T^c|\geq m/2$. Without loss of generality,  we  suppose  $|T|\geq m/2$ .
Then (\ref{eq:option2}) implies that
\begin{equation}\label{eq:option3}
 \|A_T\hat{x}-A_Tx_0\|_2^2\leq\|A_T\hat{x}-A_Tx_0\|_2^2+\|A_{T^c}\hat{x}+A_{T^c}x_0\|_2^2\leq\epsilon^2.
\end{equation}
Combining (\ref{option1}) and (\ref{eq:option3}), we have
\begin{equation}\label{eq:set1}
 D^*\hat{x}\in \{D^*x\in \R^N: \|D^*x\|_1\leq \|D^*x_0\|_1, \|A_Tx-A_Tx_0\|_2\leq \epsilon\}.
\end{equation}
Recall that $A$ satisfies S-DRIP of order $tk$ with constants $\theta_-, \ \theta_+ \in (0:2)$. 
Here
$$
t\geq\max \{\frac{1}{2\theta_--\theta_-^2},\frac{1}{2\theta_+-\theta_+^2}\}>1.
$$
So $A_T$ satisfies DRIP of order $tk$ with
\begin{equation}\label{eq:set2}
 \delta_{tk}^{A_T}\leq\max\{1-\theta_-,\ \theta_+-1\}\leq \sqrt{\frac{t-1}{t}}.
\end{equation}
Combining (\ref{eq:set1}), (\ref{eq:set2}) and Lemma \ref{mainlemma}, we obtain
\begin{equation*}
 \|\hat{x}-x_0\|_2\leq c_1\epsilon+c_2\frac{2\sigma_k(D^*x_0)_1}{\sqrt{k}},
\end{equation*}
where $c_1$ and $c_2$ are defined as before in the Theorem \ref{maintheorem}.

If $|T^c|\geq\frac{m}{2}$, we can get the corresponding result
\begin{equation*}
\|\hat{x}+x_0\|_2\leq  c_1\epsilon+c_2\frac{2\sigma_k(D^*x_0)_1}{\sqrt{k}}.
\end{equation*}
Then we have proved the theorem.
\end{proof}
In problem (\ref{primemodel}), suppose $ x_0\in D\Sigma_k^N $ and $ D^*x_0\in\R^N $ decays rapidly.
From Theorem \ref{maintheorem} and Corollary \ref{gauss_sdrip}, we conclude that the  $ \ell_1 $-analysis (\ref{l1 dictionary model}) can provide a stable solution to problem (\ref{primemodel}) if we use  as many as $ \mathcal{O}(k \log (n/k)) $ Gaussian random measurements. 
\section{Acknowledgments}
My deepest gratitude goes to Professor Zhiqiang Xu, my academic supervisor, for his guidance and many useful discussions.
\section{Appendix}
The following two lemmas are useful in the proof of Lemma \ref{mainlemma}. 
\begin{lemma}\label{quotelemma1}
(Sparse Representation of a Polytope \cite{CZ14,XX13}):
Suppose $\alpha>0$ is a constant and $s>0$ is an integer.
Set
$$
T(\alpha,s):=\{v\in\mathbb{R}^n: \|v\|_\infty\leq\alpha,\ \|v\|_1\leq s\alpha\}.
$$
For any $v\in\mathbb{R}^n$, set
$$
\mathit{U}(\alpha,s,v):=\{u\in\mathbb{R}^n:\textup{supp}(u)\subseteq \textup{supp}(v),\|u\|_0\leq s,\|u\|_1=\|v\|_1,\|u\|_\infty\leq\alpha\}.
$$
Then $v\in T(\alpha,s)$ if and only if $ v $ is in the convex hull of $\mathit{U}(\alpha,s,v)$. In particular, any $v\in T(\alpha,s)$ can be expressed as
\begin{align*}
v=&\sum_{i=1}^{M}\lambda_iu_i\quad \text{and}\quad 0\leq\lambda_i\leq 1,\,   \sum_{i=1}^{M}\lambda_i=1,\\
& u_i\in \mathit{U}(\alpha,s,v).
\end{align*}
\end{lemma}

\begin{lemma}\label{quotelemma2}(Lemma 5.3 in \cite{refCZ13}):
Suppose $m\geq r  $, $ a_1\geq a_2\geq\cdots\geq a_m\geq 0 $ and $ \sum_{i=1}^{r}a_i\geq\sum_{i=r+1}^{m}a_i $. Then for all $ \alpha\geq 1 $, we have
$$\sum_{j=r+1}^{m}a_j^\alpha\leq\sum_{i=1}^{r}a_i^\alpha. $$
\end{lemma}
Now we are ready to prove Lemma \ref{mainlemma}.
\begin{proof}[Proof of the Lemma \ref{mainlemma}]
We assume that the tight frame $ D\in\R^{n\times N} $ is normalized, i.e., $ DD^*=I $ and $ \|y\|_2=\|D^*y\|_2 $ for all $ y\in\mathbb{R}^n $. For a subset $ T\subseteq\{1,2,\ldots,N\} $, we denote $ D_T $ as the matrix $ D $ restricted to the columns indexed by $ T $ (replacing other columns by zero vectors).

Set $h:=\hat{x}-x_0$. Let $T_0$ denote the index set of the largest $k$ coefficients of $D^*x_0$ in magnitude. Then
\begin{align*}
\|D^*x_0\|_1+\rho\geq\|D^*\hat{x}\|_1&=\|D^*x_0+D^*h\|_1\\&=\|D^*_{T_0}x_0+D^*_{T_0}h+D^*_{T_0^c}x_0+D^*_{T_0^c}h\|_1\\&\geq\|            D^*_{T_0}x_0\|_1-\|D^*_{T_0}h\|_1-\|D^*_{T_0^c}x_0\|_1+\|D^*_{T_0^c}h\|_1,
\end{align*}
which implies
\begin{align*}
\|D^*_{T_0^c}h\|_1&\leq\|D^*_{T_0}h\|_1+2\|D^*_{T_0^c}x_0\|_1+\rho\\
&=\|D^*_{T_0}h\|_1+2\sigma_k(D^*x_0)_1+\rho.
\end{align*}
Suppose $ S_0 $ is the index set of the $k$ largest entries in absolute value of $ D^*h $. We get
\begin{align*}
 \|D^*_{S_0^c}h\|_1\leq\|D^*_{T_0^c}h\|_1&\leq\|D^*_{T_0}h\|_1+2\sigma_k(D^*x_0)_1+\rho
\\ &\leq\|D^*_{S_0}h\|_1+2\sigma_k(D^*x_0)_1+\rho.
\end{align*}
Set
\[
     \alpha:=\frac{\|D^*_{S_0}h\|_1+2\sigma_k(D^*x_0)_1+\rho}{k}.
\]
We divide $ D^*_{S_0^c }h$ into two parts $ D^*_{S_0^c }h=h^{(1)}+h^{(2)} $, where
\begin{align*}
h^{(1)}:=D^*_{S_0^c}h\cdot I_{\{i:|D^*_{S_0^c}h(i)|>\alpha/(t-1)\}}, \quad
h^{(2)}:=D^*_{S_0^c}h\cdot I_{\{i:|D^*_{S_0^c}h(i)|\leq\alpha/(t-1)\}} .
\end{align*}
Then a simple observation is that
$ \|h^{(1)}\|_1\leq\|D^*_{S_0^c}h\|_1\leq\alpha k $.
Set
\[
      \ell := |\textup{supp} (h^{(1)})|=\|h^{(1)}\|_0 .
\]
Since all non-zero entries of $ h^{(1)} $ have magnitude larger than $ \alpha/(t-1) $, we have
\[
\alpha k\geq \|h^{(1)}\|_1=\sum_{i\in \textup{supp}(h^{(1)})}|h^{(1)}(i)|\geq\sum_{i\in \textup{supp}(h^{(1)})}\frac{\alpha}{t-1}=\ell\cdot\frac{\alpha}{t-1},
\]
which implies  $ \ell\leq (t-1)k $.

Note that
\begin{align*}
\|h^{(2)}\|_1&=\|D^*_{S_0^c}h\|_1-\|h^{(1)}\|_1\leq k\alpha-\ell \cdot \frac{ \alpha}{t-1}=(k(t-1)-\ell)\frac{\alpha}{t-1},\\
\|h^{(2)}\|_\infty&\leq \frac{\alpha}{t-1}.
\end{align*}
Then in Lemma \ref{quotelemma1}, by setting $ s:=k(t-1)-\ell $, we can express $ h^{(2)} $ as a weighted mean:
$$
  h^{(2)}=\sum_{i=1}^{M}\lambda_iu_i,
$$
where $ 0\leq \lambda_i\leq 1 $, $ \sum_{i=1}^{M}\lambda_i=1$, $ \|u_i\|_0\leq k(t-1)-\ell $, $\|u_i\|_\infty\leq\alpha/(t-1) $ and  $\textup{supp}(u_i)\subseteq \textup{supp}(h^{(2)}) $.
Thus
\begin{align*}
\|u_i\|_2\leq\sqrt{\|u_i\|_0}\cdot \|u_i\|_\infty &
=\sqrt{k(t-1)-\ell }\cdot \|u_i\|_\infty\\
& \leq\sqrt{k(t-1)}\cdot \|u_i\|_\infty\\
& \leq\alpha\sqrt{k/(t-1)}.
\end{align*}
Recall that  $\alpha=\frac{\|D^*_{S_0}h\|_1+2\sigma_k(D^*x_0)_1+\rho}{k}$.
Then
\begin{align}
\|u_i\|_2&\leq\alpha\sqrt{k/(t-1)} \nonumber\\
&\leq\frac{\|D^*_{S_0}h\|_2}{\sqrt{t-1}}+\frac{2\sigma_k(D^*x_0)_1+\rho}{\sqrt{k(t-1)}}\nonumber\\
&\leq\frac{\|D^*_{S_0}h+h^{(1)}\|_2}{\sqrt{t-1}}+\frac{2\sigma_k(D^*x_0)_1+\rho}{\sqrt{k(t-1)}}\nonumber\\
&=\frac{z+R}{\sqrt{t-1}},\label{opt-u}
\end{align}
where   $ z:=\|D^*_{S_0}h+h^{(1)}\|_2, \,\, R:=\frac{2\sigma_k(D^*x_0)_1+\rho}{\sqrt{k}}$.

Now we suppose $ 0\leq\mu\leq 1 $, $ d\geq 0$  are two constants to be determined.  Set
\[
   \beta_j:=D^*_{S_0}h+h^{(1)}+\mu\cdot u_j,\,\, j=1,\ldots,M.
\]
Then for any fixed $i\in[M]$,
\begin{align}\label{expression1}
\sum_{j=1}^{M}\lambda_j\beta_j-d\beta_i&=D^*_{S_0}h+h^{(1)}+\mu\cdot h^{(2)}-d\beta_i\\\nonumber
&=(1-\mu-d)(D^*_{S_0}h+h^{(1)})-d\mu u_i+\mu D^*h.
\end{align}
For $\sum_{i=1}^M\lambda_i=1$, we have the following identity
\begin{equation}\label{eq:lamdabeta2}
(2d-1)\sum_{1\leq i<j\leq M}\lambda_i\lambda_j\|AD(\beta_i-\beta_j)\|_2^2=
\sum_{i=1}^{M}\lambda_i\|AD(\sum_{j=1}^{M}\lambda_j\beta_j-d\beta_i)\|_2^2-\sum_{i=1}^{M}\lambda_i(1-d)^2\|AD\beta_i\|_2^2.
\end{equation}
In (\ref{expression1}), we chose $ d=1/2 $ and $ \mu=\sqrt{t(t-1)}-(t-1)< 1/2 $. Then (\ref{eq:lamdabeta2}) implies
\begin{align} 0&=\sum_{i=1}^{M}\lambda_i\|AD(\sum_{j=1}^{M}\lambda_j\beta_j-d\beta_i)\|_2^2-\sum_{i=1}^{M}\frac{\lambda_i}{4}\|AD\beta_i\|_2^2\nonumber \nonumber\\
&\overset{(\ref{expression1})}{=}\sum_{i=1}^{M}\lambda_i\|AD\left( (\frac{1}{2}-\mu)(D^*_{S_0}h+h^{(1)})-\frac{\mu}    {2}u_i+\mu D^*h\right) \|_2^2-\sum_{i=1}^{M}\frac{\lambda_i}{4}\|AD\beta_i\|_2^2\nonumber \\
&=\sum_{i=1}^{M}\lambda_i\|AD\left( (\frac{1}{2}-\mu)(D^*_{S_0}h+h^{(1)})-\frac{\mu}{2}u_i\right) \|_2^2\nonumber\\
&\qquad +2\left\langle AD\left( (\frac{1}{2}-\mu)(D^*_{S_0}h+h^{(1)})-\frac{\mu}{2}h^{(2)}\right), \mu ADD^*h\right\rangle
+\mu^2\|ADD^*h\|_2^2-\sum_{i=1}^{M}\frac{\lambda_i}{4}\|AD\beta_i\|_2^2 \nonumber\\
&=\sum_{i=1}^{M}\lambda_i\|AD\left( (\frac{1}{2}-\mu)(D^*_{S_0}h+h^{(1)})-\frac{\mu}{2}u_i\right) \|_2^2 \label{eq:threeterm}\\
&\qquad +\mu(1-\mu)\left\langle
AD(D^*_{S_0}h+h^{(1)}),ADD^*h\right\rangle -\sum_{i=1}^{M}\frac{\lambda_i}{4}\|AD\beta_i\|_2^2.\nonumber
\end{align}
We next estimate the three terms in (\ref{eq:threeterm}).
First we give the following useful relation:
\begin{align}
\left\langle D(D^*_{S_0}h+h^{(1)}), Dh^{(2)}\right\rangle
&=\left\langle D(D^*_{S_0}h+h^{(1)}), D(D^*h-D^*_{S_0}h-h^{(1)})\right\rangle \nonumber\\
&=\left\langle D(D^*_{S_0}h+h^{(1)}), h\right\rangle-\left\langle D(D^*_{S_0}h+h^{(1)}),D(D^*_{S_0}h+h^{(1)})\right\rangle\nonumber\\
&=\left\langle D^*_{S_0}h+h^{(1)}, D^*h\right\rangle-\|D(D^*_{S_0}h+h^{(1)})\|_2^2\nonumber\\
&=\|D^*_{S_0}h+h^{(1)}\|_2-\|D(D^*_{S_0}h+h^{(1)})\|_2^2.\label{usefulrelation}
\end{align}
Noting that $\| D^*_{S_0}h\|_0\leq k $, $ \|h^{(1)}\|_0= \ell\leq(t-1)k $ and  $\|u_i\|_0\leq s =k(t-1)-\ell$, we obtain
\[
\|D^*_{S_0}h+h^{(1)}\|_0\leq\ell+k\leq t\cdot k,\quad \| \beta_i\|_0\leq  \| D^*_{S_0}h\|_0 +  \|h^{(1)}\|_0+ \|u_i\|_0\leq t\cdot k,
\]
and
$$
\|(\frac{1}{2}-\mu)(D^*_{S_0}h+h^{(1)})-\frac{\mu}{2}u_i\|_0\leq t\cdot k.
$$
Here we assume $ t\cdot k $ as an integer first. Since $A$ satisfies the DRIP of order $t\cdot k$ with constant $\delta $, we can obtain
\begin{align*}
&\quad \sum_{i=1}^{M}\lambda_i \|AD\left( (\frac{1}{2}-\mu)(D^*_{S_0}h+h^{(1)})-\frac{\mu}{2}u_i\right) \|_2^2\\
&\leq\sum_{i=1}^{M}\lambda_i(1+\delta)\|D\left((\frac{1}{2}-\mu)(D^*_{S_0}h+h^{(1)})-\frac{\mu}{2}u_i\right)\|_2^2\\
&=(1+\delta)\left((\frac{1}{2}-\mu)^2 \|D(D^*_{S_0}h+h^{(1)})\|_2^2+\frac{\mu^2}{4}\sum_{i=1}^{M}\lambda_i\|Du_i\|_2^2-\mu(\frac{1}{2}-\mu)\left\langle D(D^*_{S_0}h+h^{(1)}), Dh^{(2)}\right\rangle\right)\\
&\overset{(\ref{usefulrelation})}{=}(1+\delta)\left(\frac{1}{2}(\frac{1}{2}-\mu) \|D(D^*_{S_0}h+h^{(1)})\|_2^2+\frac{\mu^2}{4}\sum_{i=1}^{M}\lambda_i\|Du_i\|_2^2-\mu(\frac{1}{2}-\mu) \|D^*_{S_0}h+h^{(1)}\|_2^2\right),
\end{align*}
\begin{equation*}
\begin{split}
\left\langle AD(D^*_{S_0}h+h^{(1)}), ADD^*h\right\rangle
&=\left\langle AD(D^*_{S_0}h+h^{(1)}), Ah\right\rangle\\
&\leq\sqrt{1+\delta}\cdot \|D(D^*_{S_0}h+h^{(1)})\|_2\cdot\epsilon
\end{split}
\end{equation*}
and
\begin{align*}
&\quad \sum_{i=1}^{M}\lambda_i\|AD\beta_i\|_2^2\\
&=\sum_{i=1}^{M}\lambda_i\|AD(D^*_{S_0}h+h^{(1)}+\mu\cdot u_i)\|_2^2\\
&\geq(1-\delta)\sum_{i=1}^{M}\lambda_i\|D(D^*_{S_0}h+h^{(1)}+\mu\cdot u_i)\|_2^2\\
&=(1-\delta)\left( \|D(D^*_{S_0}h+h^{(1)})\|_2^2+\mu^2\sum_{i=1}^{M}\lambda_i\|Du_i\|_2^2+2\mu\left\langle D(D^*_{S_0}h+h^{(1)}), Dh^{(2)}\right\rangle \right)\\
&\overset{(\ref{usefulrelation})}{=}(1-\delta)\left( (1-2\mu)\|D(D^*_{S_0}h+h^{(1)})\|_2^2+\mu^2\sum_{i=1}^{M}\lambda_i\|Du_i\|_2^2+2\mu\|D^*_{S_0}h+h^{(1)}\|_2^2 \right).
\end{align*}
Combining the above results with (\ref{opt-u}) and (\ref{eq:threeterm}), we get
\vspace{-2em}
\begin{align*}
0&\leq\frac{1}{2}(1+\delta)(\frac{1}{2}-\mu) \|D(D^*_{S_0}h+h^{(1)})\|_2^2+\frac{1+\delta}{4}\mu^2\sum_{i=1}^{M}\lambda_i\|Du_i\|_2^2-(1+\delta)\mu(\frac{1}{2}-\mu) \|D^*_{S_0}h+h^{(1)}\|_2^2 \\
&\quad+\mu(1-\mu)\sqrt{1+\delta}\|D(D^*_{S_0}h+h^{(1)})\|_2\cdot\epsilon\\
&\quad-\frac{1}{4}(1-\delta)(1-2\mu)\|D(D^*_{S_0}h+h^{(1)})\|_2^2-\frac{1-\delta}{4}\mu^2\sum_{i=1}^{M}\lambda_i\|Du_i\|_2^2-\frac{1-\delta}{2}\mu\|D^*_{S_0}h+h^{(1)}\|_2^2 \\
&=\delta(\frac{1}{2}-\mu)\|D(D^*_{S_0}h+h^{(1)})\|_2^2+(\mu^2(1+\delta)-\mu)\|D^*_{S_0}h+h^{(1)}\|_2^2+\frac{\delta}{2}\mu^2\sum_{i=1}^{M}\lambda_i\|Du_i\|_2^2\\
&\quad+\mu(1-\mu)\sqrt{1+\delta}\|D(D^*_{S_0}h+h^{(1)})\|_2\cdot\epsilon\\
&\leq(\delta(\frac{1}{2}-\mu)+\mu^2(1+\delta)-\mu)z^2+\frac{\delta}{2}\mu^2\sum_{i=1}^{M}\lambda_i\|u_i\|_2^2
+\mu(1-\mu)\sqrt{1+\delta}\cdot z\cdot\epsilon\\
&\overset{(\ref{opt-u})}{\leq} \left((1+\delta)(\frac{1}{2}-\mu)^2-\frac{1-\delta}{4} \right)z^2+\frac{\delta}{2}\mu^2\frac{(z+R)^2}{t-1} +\mu(1-\mu)\sqrt{1+\delta}\cdot z\cdot\epsilon \\
&=\left((\mu^2-\mu)+\delta\left( \frac{1}{2}-\mu+(1+\frac{1}{2(t-1)})\mu^2\right)  \right)z^2+\left( \mu(1-\mu)\sqrt{1+\delta}\cdot\epsilon+\frac{\delta\mu^2R}{t-1}\right)z+\frac{\delta\mu^2R^2}{2(t-1)} \\
&=-t\left((2t-1)-2\sqrt{t(t-1)} \right) (\sqrt{\frac{t-1}{t}}-\delta)z^2+\left( \mu^2\sqrt{\frac{t}{t-1}}\sqrt{1+\delta}\cdot\epsilon+\frac{\delta\mu^2R}{t-1}\right)z+\frac{\delta\mu^2R^2}{2(t-1)}\\
&=\frac{\mu^2}{t-1}\left(-t(\sqrt{\frac{t-1}{t}}-\delta)z^2+(\sqrt{t(t-1)(1+\delta)}\epsilon+\delta R)z+\frac{\delta R^2}{2} \right),
\end{align*}
which is a quadratic inequality for $z$. Recall that $ \delta<\sqrt{(t-1)/t} $. So by solving the above inequality we get
\begin{align*}
z&\leq\frac{(\sqrt{t(t-1)(1+\delta)}\epsilon+\delta R)+\left((\sqrt{t(t-1)(1+\delta)}\epsilon+\delta R)^2+2t(\sqrt{(t-1)/t}-\delta)\delta R^2 \right)^{1/2}  }
{2t(\sqrt{(t-1/t)}-\delta) }\\
&\leq\frac{\sqrt{t(t-1)(1+\delta)}}{t(\sqrt{(t-1)/t}-\delta)}\epsilon+\frac{2\delta+\sqrt{2t(\sqrt{(t-1)/t}-\delta)\delta}}{2t(\sqrt{(t-1)/t}-\delta)}R.
\end{align*}
We know  $ \|D^*_{S_0^c}h\|_1\leq\|D^*_{S_0}h\|_1+R\sqrt{k} $. In the Lemma \ref{quotelemma2}, if we set $ m=N $, $ r=k $, $ \lambda=R\sqrt{k}\geq 0 $ and $ \alpha=2 $, 
we can obtain
$$
\|D^*_{S_0^c}h\|_2\leq\|D^*_{S_0}h\|_2+R.
$$
So
\begin{align*}
\|h\|_2&=\|D^*h\|_2\\
&=\sqrt{\|D^*_{S_0}h\|_2^2+\|D^*_{S_0^c}h\|_2^2}\\
&\leq\sqrt{\|D^*_{S_0}h\|_2^2+(\|D^*_{S_0}h\|_2+R)^2}\\ &\leq\sqrt{2\|D^*_{S_0}h\|_2^2}+R\leq\sqrt{2}z+R\\
&\leq\frac{\sqrt{2(1+\delta)}}{1-\sqrt{t/(t-1)}\delta}\epsilon+\left( \frac{\sqrt{2}\delta+\sqrt{t(\sqrt{(t-1)/t}-\delta)\delta}}{t(\sqrt{(t-1)/t}-\delta)}+1 \right) R.
\end{align*}
Substituting $ R $ into this inequality, we can get the conclusion.
For the case where $t\cdot k$ is not an integer, we set $t^*:=\lceil tk\rceil / k$, then $t^*>t$ and $\delta_{t^*k}=\delta_{tk}<\sqrt{\frac{t-1}{t}}<\sqrt{\frac{t^*-1}{t^*}}$. We can prove the result by working on $\delta_{t^*k}$.
\end{proof}

\end{document}